\documentclass{amsart}

\usepackage{blkarray}

\newtheorem{theorem}{Theorem}[section]
\newtheorem{lemma}[theorem]{Lemma}
\newtheorem{proposition}[theorem]{Proposition}
\theoremstyle{definition}
\newtheorem{example}[theorem]{Example}

\theoremstyle{remark}
\newtheorem{remark}[theorem]{Remark}

\numberwithin{equation}{section}

\usepackage{hyperref}
\hypersetup{
    colorlinks=true,
    urlcolor=black,
    linkcolor=blue,
    breaklinks=true
}

\begin{document}

\title{Further results on the real cubic truncated moment problem}


\author{Abdelaziz El Boukili}
\address{Laboratory of Analysis, Geometry and Applications (LAGA),
  Department of Mathematics,
  Faculty of Sciences, Ibn Tofail University,
  Kenitra, B.P. 133, Morocco}
\curraddr{}
\email{abdelaziz.elboukili@uit.ac.ma}
\thanks{}

\author{Amar Rhazi}
\address{Laboratory of Analysis, Geometry and Applications (LAGA),
  Department of Mathematics,
  Faculty of Sciences, Ibn Tofail University,
  Kenitra, B.P. 133, Morocco}
\curraddr{}
\email{amar.rhazi@uit.ac.ma}
\thanks{}


\author{Bouazza El Wahbi}
\address{Laboratory of Analysis, Geometry and Applications (LAGA),
  Department of Mathematics,
  Faculty of Sciences, Ibn Tofail University,
  Kenitra, B.P. 133, Morocco}
\curraddr{}
\email{bouazza.elwahbi@uit.ac.ma}
\thanks{}



\date{}

\dedicatory{}

\begin{abstract}
In this paper, we devote our interest to solving the real cubic truncated moment problem. We provide some results that allow to get a complete solution via a minimal representing measure. Some numerical examples are also  presented to emphasize the simplicity of our approach.
\end{abstract}

\maketitle

\section{Introduction}\label{MysecLabe1}
Let $ \beta \equiv \beta^{(m)} = \left \{\beta_{ij} \right \}_{i, j \in \mathbb{Z}_+, 0 \leq i + j \leq m} = \left \{\beta_{00},  \beta_{10}, \beta_{01}, \ldots, \beta_{m0}, \ldots, \beta_{0m} \right \} $ with $ \beta_{00}> 0 $ be a doubly indexed finite sequence of real numbers. The truncated real moment problem (TRMP) associated to $ \beta $ consists in finding the existence of a Borel positive  measure $ \mu $ supported in $ \mathbb{R}^2 $ such that
\begin{equation}
\beta_{i j}=\int x^{i} y^{j} d\mu, \quad (i,j\in \mathbb{Z}_+,~\  \leq i+j \leq m).\label{MyEqLabe1-1}
\end{equation}
The measure $ \mu $  in \eqref{MyEqLabe1-1} is called a representing measure for $\left \{\beta_{ij} \right \}_{i, j \in \mathbb{Z}_+, 0 \leq i + j \leq m}$, and the sequence $\beta $ a truncated moment sequence.

The truncated complex moment problem (TCMP) for a doubly indexed finite sequence $ s \equiv s^{(m)} = (s_{ij})_{i, j \in \mathbb{Z}_{+}, 0 \leq i + j \leq m} = \{s_{00}, s_{01}, s_{10}, \ldots, s_{0m}, \ldots, s_{m0} \} $ of complex numbers with $ s_{00}> 0 $ and $ s_{ji} = \overline{s_{ij}} $ concerns the existence of a positive Borel measure $ \sigma $ supported on $ \mathbb{C}$ such that, 
\begin{gather*}
s_{ij} = \int \bar{z}^{i}z^{j} d\sigma, \quad (i,j \in \mathbb{Z}_{+},~\ 0 \leq i+j \leq m).
\end{gather*}

C. Bayer and J. Teichmann \cite{BaTe}  proved that if a sequence of moments admits one or more representing measures, one of these must be of an atomic finite type. 

So if a real doubly indexed finite sequence $ \beta^{(m)} $ has a representing measure, it admits a finite atomic representing measure $ \mu $, i.e. we can write $ \mu: = \sum\limits_{k = 1}^{r} \rho_{k} \delta_{(x_k, y_k)} $ where the positive numbers $ \rho_{k} $ and the couples $ (x_k, y_k) $, $ 1 \leq k \leq r $ are called weights and atoms respectively of the measure $ \mu $ which is said to be $r$-atomic, and we have
$$ \beta_{ij} = \rho_{1} x_{1}^{i} y_{1}^{j} + \cdots + \rho_{r} x_{r}^{i} y_{r}^{j} = \int x^{i} y^{j} d \mu, \quad 0 \leq i + j \leq m. $$

Curto and Fialkow have shown in \cite[Proposition 1.12]{CuFi4} an equivalence between TRMP and TCMP in the case where $ m $ is even. So, we can talk simply of the truncated moment problem (TMP). In \cite{CuFi1, CuFi4, CuYo2, FiaNi}, the authors provide   solutions for the TMP when $ m = 2 $ and $ m = 4 $. Their approaches are generally based on positivity and the flat extension of the moment matrix $ \mathcal{M}(n)= (\beta_{i+j})_{i, j \in \mathbb{Z}_{+}, 0 \leq i + j \leq 2 n} $ associated to the sequence $ \beta $. For some even  values of $ m $ greater than 4, Curto and Fialkow \cite{CuFi6, CuFiM} used the theory of recursively generated and (or) recursively determined moment matrices. While for the case $ m = 6 $, it has been closely studied by Curto et al in \cite{CuFi6, CuFiM} and by Yoo \cite{CuYo1, Yoo1} in non-extremal case, where the rank of the associated moment matrix is strictly lower to the cardinal of the associated algebraic variety, and in extremal case when the rank and the cardinal of the algebraic variety are equal.

 For the odd cases, D. Kimsey stated in \cite{Kim1, Kim2} a complete solution of the cubic TCMP ($m =3 $). Based on the commutativity conditions of the matrices associated with the cubic moment sequence, he showed that there is a representative measure at most $4$ atomic. While for quintic TCMP ($m =5$), we find in \cite{EHIZ} an incomplete solution with a remaining case.

In \cite{CuYo4}, Curto and Yoo presented an alternative solution of the nonsingular cubic TRMP (i.e. $ \mathcal{M}(1)> 0 $) using the invariance under a degree-one transformation, positivity, flatness and recursively determined moment matrices.
 
In this article, we aim to provide a simple and complete alternative solution to the real cubic moment problem.

Let $ \beta \equiv \beta^{(3)} = \left \{\beta_{ij} \right \}_{0 \leq i + j \leq 3} $ be a doubly indexed sequence with real values given with $ \beta_{00}>0 $. As $ m $ is odd ($ m = 3 $), we group the data of the sequence $ \beta $ in the following two matrices,
 \begin{equation}
 \mathcal{M}(1):=\left(\begin{array}{ccc}
\beta_{00} & \beta_{10} & \beta_{01} \\
\beta_{10} & \beta_{20} & \beta_{11} \\
\beta_{01} & \beta_{11} & \beta_{02} 
 \end{array}\right)\text{ and } B(2):=\left(\begin{array}{ccc}
\beta_{20} & \beta_{11} & \beta_{02} \\
\beta_{30} & \beta_{21} & \beta_{12} \\
\beta_{21} & \beta_{12} & \beta_{03} 
\end{array}\right).\label{MyEqLabe1-2}
\end{equation} 
Then, we determine quartic moments $ \beta_{40}, \beta_{31}, \beta_{22}, \beta_{31}, \beta_{22}, \beta_{13}$ and $ \beta_{04}$ to construct a positive semidefinite extension $ \mathcal{M}(2)$ of the matrix $ \mathcal{M}(1)$ as follows,
 \begin{equation}
\mathcal{M}(2): \begin{pmatrix} 
\beta_{00} & \vert & \beta_{10} & \beta_{01} & \vert & \beta_{20} & \beta_{11} & \beta_{02}\\
     --    & - &      --    &      --    & - &      --    &      --    &    --    \\
\beta_{10} & \vert & \beta_{20} & \beta_{11} & \vert & \beta_{30} & \beta_{21} & \beta_{12}\\
\beta_{01} & \vert & \beta_{11} & \beta_{02} & \vert & \beta_{21} & \beta_{12} & \beta_{03}\\
     --    & - &      --    &      --    & - &      --    &      --    &    --    \\
\beta_{20} & \vert & \beta_{30} & \beta_{21} & \vert & \beta_{40} & \beta_{31} & \beta_{22}\\
\beta_{11} & \vert & \beta_{21} & \beta_{12} & \vert & \beta_{31} & \beta_{22} & \beta_{13}\\
\beta_{02} & \vert & \beta_{12} & \beta_{03} & \vert & \beta_{22} & \beta_{13} & \beta_{04}
 \end{pmatrix},\label{MyEqLabe1-3}
\end{equation} 
so that $ \operatorname{rank} \mathcal{M}(2) = \operatorname{rank} \mathcal{M}(1) $ or in the opposite case, $\mathcal{M}(2)$ can be extended to a positive semidefinite matrix $ \mathcal{M}(3)$ by calculating quintic moments ($ \beta_{50} $, $ \beta_{41} $, $ \beta_{32} $, $ \beta_{23} $, $ \beta_{14} $ and $ \beta_{05} $), and sixtics ($ \beta_{60} $, $ \beta_{51} $, $ \beta_{42} $, $ \beta_{33} $, $ \beta_{24} $, $ \beta_{15} $ and $ \beta_{06} $),
$$
\mathcal{M}(3)=\left(\begin{array}{ccccccccccccc}
\beta_{00} & \vert & \beta_{10} & \beta_{01}  & \vert &  \beta_{20} & \beta_{11} & \beta_{02} & \vert &  \beta_{30} & \beta_{21} & \beta_{12} & \beta_{03} \\
    --    &-&      --    &      --    & - &      --    &      --    &    --      & - &      --    &      --    &    -- &    --\\
\beta_{10} & \vert &  \beta_{20} & \beta_{11} & \vert &  \beta_{30} & \beta_{21} & \beta_{12} & \vert &  \beta_{40} & \beta_{31} & \beta_{22} & \beta_{13} \\
\beta_{01} & \vert &  \beta_{11} & \beta_{02} & \vert &  \beta_{21} & \beta_{12} & \beta_{03} & \vert &  \beta_{31} & \beta_{22} & \beta_{13} & \beta_{04} \\
 --    &-&      --    &      --    & - &      --    &      --    &    --      & - &      --    &      --    &    -- &    --\\
\beta_{20} & \vert &  \beta_{30} & \beta_{21} & \vert &  \beta_{40} & \beta_{31} & \beta_{22} & \vert &  \beta_{50} & \beta_{41} & \beta_{32} & \beta_{23} \\
\beta_{11} & \vert &  \beta_{21} & \beta_{12} & \vert &  \beta_{31} & \beta_{22} & \beta_{13} & \vert &  \beta_{41} & \beta_{32} & \beta_{23} & \beta_{14} \\
\beta_{02} & \vert &  \beta_{12} & \beta_{03} & \vert &  \beta_{22} & \beta_{13} & \beta_{04} & \vert &  \beta_{32} & \beta_{23} & \beta_{14} & \beta_{05} \\
 --    &-&      --    &      --    & - &      --    &      --    &    --      & - &      --    &      --    &    -- &    --\\
\beta_{30} & \vert & \beta_{40} & \beta_{31}  & \vert &  \beta_{50} & \beta_{41} & \beta_{32} & \vert &  \beta_{60} & \beta_{51} & \beta_{42} & \beta_{33} \\
\beta_{21} & \vert &  \beta_{31} & \beta_{22} & \vert &  \beta_{41} & \beta_{32} & \beta_{23} & \vert &  \beta_{51} & \beta_{42} & \beta_{33} & \beta_{24} \\
\beta_{12} & \vert &  \beta_{22} & \beta_{13} & \vert & \beta_{32} & \beta_{23} & \beta_{14}  & \vert &  \beta_{42} & \beta_{33} & \beta_{24} & \beta_{15} \\
\beta_{03} & \vert &  \beta_{13} & \beta_{04} & \vert &  \beta_{23} & \beta_{14} & \beta_{05} & \vert &  \beta_{33} & \beta_{24} & \beta_{15} & \beta_{06}\\ 
\end{array}\right),
$$
such that  $\operatorname{rank}\mathcal{M}(3)=\operatorname{rank}\mathcal{M}(2)$.

The remainder of this paper is organized as follows. In Section \ref{MysecLabe2}, we state some notations and some tools which will be used for solving the problems of the truncated moments. In Section \ref{MysecLabe3}, we present our main results illustrated by numerical examples.

\section{Preliminaries}\label{MysecLabe2}

In this section, we recall some results and notations that will be used in the sequel. 

We denote by $ M_{(p, q)} (\mathbb{K}) $, where $ \mathbb{K} = \mathbb{R} $ or $ \mathbb{C} $, the set of $ p \times q $ matrices and $\mathbb{R} [x, y] $, the space of polynomials with two indeterminates. $ \mathcal{P}_n $ will stand for the space of polynomials with two indeterminates, and real coefficients with total degree is lower than or equal to $ n $.

For a symmetric matrix $ A $, we write $ A \succeq 0 $ if $ A $ is positive semidefinite and $ A> 0 $ if $ A $ is positive definite.

To a sequence of moments $ \beta = \beta^{(2 n)} \equiv \left \{\beta_{ij} \right \}_{i + j \leq 2n} $, we associate the matrix moment $ \mathcal{M}(n) $, and if $ \mu $ is a representing measure  of $ \beta $ then for any polynomial $ P \equiv \sum\limits_{l, k} a_{lk} x^{l} y^{k} \in \mathbb{R} [x, y] $, we have,
   $$ 0 \leq \int|P(x,y)|^{2} d \mu=\sum_{l, k, l^{\prime}, k^{\prime}} a_{lk} a_{l ^{\prime} k^{\prime}} \int x^{l+k^{\prime}} y^{k+l^{\prime}}d \mu=\sum_{l, k, l ^{\prime}, k^{\prime}} a_{lk} a_{l^{\prime} k^{\prime}} \beta_{l+k^{\prime}, k+l^{\prime}}.$$
Hence, if $ \beta $ admits a representing measure, then the matrix $ \mathcal{M}(n) $ is positive semidefinite. The matrix $ \mathcal{M}(n) $ admits a decomposition by blocks $ \mathcal{M}(n) = (B [i, j])_{0 \leq i, j \leq n} $,
$$
\mathcal{M}(n):=\left(\begin{array}{cccc}
B[0,0] & B[0,1] & \dots & B[0, n] \\
B[1,0] & M[1,1] & \dots & B[1, n] \\
\vdots & \vdots & \ddots & \vdots \\
B[n, 0] & B[n, 1] & \dots & B[n, n]
\end{array}\right)
$$
where,
$$B[i, j]=\left(\begin{array}{cccc}
\beta_{i+j,0} & \beta_{i+j-1, 1} & \dots & \beta_{i,j} \\
\beta_{i+j-1,1 } & \beta_{i+j-2, 2} & \dots & \beta_{i-1,j+1} \\
\vdots & \vdots & \ddots & \vdots \\
\beta_{j, i} & \beta_{j-1, i+1} & \dots & \beta_{0, i+j}
\end{array}\right), \quad 0\leq i,j\leq n.$$

Thus, each block $ B[i,j] $ has the Hankel's property, i.e it is constant on each cross diagonal. Furthermore if we choose a labeling for the columns and rows of the moment matrix $ \mathcal{M}(n) $ by considering the lexicographic order of the monomials in degree, $1$, $X$, $Y$, $X^{2}$, $XY$, $Y^{2},\ldots$, $X^{n}$, $X^{n-1}Y,\ldots$, $XY^{n-1}$, $Y^{n}$, then as an example, the matrix $ \mathcal{M}(2) $ is written as,
\begin{equation}
\mathcal{M}(2)=\begin{blockarray}{ccccccccc}
{} & 1 & {} & X & Y & {} & X^{2} & X Y & Y^{2} \\
\begin{block}{c(cccccccc)}
1 & \beta_{00} & \vert & \beta_{10} & \beta_{01} & \vert & \beta_{20} & \beta_{11} & \beta_{02} \\
{} & -- & - & -- & -- & - & -- & -- & --  \\
X &\beta_{10} & \vert & \beta_{20} & \beta_{11} & \vert & \beta_{30} & \beta_{21} & \beta_{12}  \\
Y &\beta_{01} & \vert & \beta_{11} & \beta_{02} & \vert & \beta_{21} & \beta_{12} & \beta_{03}  \\
{} & -- & - & -- & -- & - & -- & -- & --  \\
X^{2} &\beta_{20} & \vert & \beta_{30} & \beta_{21} & \vert & \beta_{40} & \beta_{31} & \beta_{22}  \\
XY &\beta_{11} & \vert & \beta_{21} & \beta_{12} & \vert & \beta_{31} & \beta_{22} & \beta_{13}   \\
Y^{2} &\beta_{02} & \vert & \beta_{12} & \beta_{03} & \vert & \beta_{22} & \beta_{13} & \beta_{04}   \\
\end{block}
\end{blockarray}.\label{MyEqLabe2-1}
\end{equation}

In the following theorem, Smul'jan \cite{Smu} establishes a necessary and sufficient condition which ensures the positive extension and the flatness of a positive semidefinite matrix.
\begin{theorem}\label{MythmLabe2-1}
Let $ A \in \mathcal{M}_{(n, n)} (\mathbb{C}) $, $ B \in \mathcal{M}_{(n, p)} (\mathbb{C}) $, and $ C \in \mathcal{M}_{(p, p)} ( \mathbb{C}) $ be matrices of complex numbers. We have,
$$
\tilde{A}=\left(\begin{array}{cc}
A & B \\
B^{*} & C
\end{array}\right) \succeq 0 \Longleftrightarrow
\left\{\begin{array}{lll}
& A \succeq 0 \\
& B=A W\ (\text { for some } W\in \mathcal{M}_{(n,p)}(\mathbb{C}) ). \\
& C \succeq W^{*} A W
\end{array}\right.
$$
Moreover
$$\operatorname{rank}(\tilde{A})=\operatorname{rank}(A)\Longleftrightarrow C=W^{*}AW  \text{ for some }  W  \text{ such that } AW=B.$$
\end{theorem}
When $\tilde{A}$ in Theorem \ref{MythmLabe2-1} has the same rank as $ A $, we say that $ \tilde{A} $ is a flat extension of $ A $. Moreover, if $ A \succeq 0 $ then each flat extension $ \tilde{A} $ of $ A $ is positive semidefinite.
\begin{remark} \label{MyremLabe2-2}~\
\begin{enumerate}
\item According to the factorization lemma of Douglas \cite{Doug}, the condition $ B = AW $ for a certain matrix $ W $ is equivalent to $ \operatorname{Ran}(B) \subseteq \operatorname{Ran}(A) $.
\item Since $A = A^{*}$, we obtain $W^{*}AW$ independent of $W$ provided that $B = AW$.
\end{enumerate}
\end{remark}
According to the Theorem \ref{MythmLabe2-1}, $\mathcal{M}(n) \succeq 0 $ admits a flat positive semidefinite extension
\begin{equation}
\mathcal{M}(n+1)=\left(\begin{array}{cc}
\mathcal{M}(n) & B(n+1) \\
B(n+1)^{T} & C(n+1)
\end{array}\right),\label{MyEqLabe2-2}
\end{equation} 
is equivalent to have the following two conditions,
\begin{enumerate}
\item $B(n+1)=\mathcal{M}(n)W$ for a matrix $W$;
\item $C(n+1)=W^{T}\mathcal{M}(n)W$ is a Hankel matrix.
\end{enumerate}
Let us notice also that we have
\begin{equation}
\left(\begin{array}{cc}
I_p & 0 \\
-W^{T} & I_q
\end{array}\right) \mathcal{M}(n+1)\left(\begin{array}{cc}
I_p & -W \\
0 & I_q
\end{array}\right)=\left(\begin{array}{cc}
\mathcal{M}(n) & 0 \\
0 & C(2)-W^{T} \mathcal{M}(n) W
\end{array}\right),\label{MyEqLabe2-3}
\end{equation}
where $ I_{p} $ and $ I_{q} $ are the unit matrices of respective orders $ p = n + 2 $ and $ q~=~ \dfrac{(n + 1)(n + 1)}{2} $. So from \eqref{MyEqLabe2-3}, we deduce that,
\begin{equation}
\operatorname{rank}\mathcal{M}(n+1)=\operatorname{rank}\mathcal{M}(n)+\operatorname{rank}\left(C(2)-W^{T} \mathcal{M}(n) W\right).\label{MyEqLabe2-4}
\end{equation}
We consider the Riesz functional $ L_{\beta}: \mathcal{P}_{2n} \longrightarrow \mathbb{R} $ defined by
$$ L_{\beta} \left (P = \sum_{0 \leq i + j \leq 2n} a_{ij} x^{i} y^{j} \right) = \sum_{0 \leq i + j \leq 2n} a_{ij} \beta_{ij}.$$

It is easy to see that if $ \hat{P} = (a_{ij}) $ and $ \hat {Q} = (b_{ij}) $ are respectively the column vectors of the polynomials $ P $ and $ Q $ in the basis of $ \mathcal{P}_{n} $  formed by monomials in lexicographical order in degrees $ 1, x, y, x^{2}, xy, y^{2}, \cdots, x^{n}, \cdots, y^{n} $, then the action of the matrix $ \mathcal{M}(n) $ on the polynomials $ P $ and $ Q $ is given by 
$$ \left \langle \mathcal{M}(n) \hat{P}, \hat{Q} \right \rangle: = L_{\beta} (PQ), \left (P, Q \in \mathcal{P}_{n} \right).$$
 
Therefore, the entry of the matrix $ \mathcal{M}(n) $ related to the row $ X^{k}Y^{l} $ and the column $ X^{k^{'}} Y^{l^{'}} $ is 
$$ \beta_{k^{'}+ k, l^{'} + l} = \left \langle X^{k^{'}} Y^{ l^{'}}, X^{k} Y^{l} \right \rangle.$$

The correspondence between $ \mathcal{P}_{n} $ and $ \mathcal{C}_{\mathcal{M}(n)} $, the column space of the matrix $ \mathcal{M}(n) $, is given by $ P(X, Y) = \mathcal{M}(n) \hat{P} $ where $ P = \sum\limits_{0 \leq i + j \leq 2n} a_{ij}x^{i}y^{j} $, that is, $ P(X, Y) $ is a linear combination of $ \mathcal{M}(n)$ columns.

Considering $ \mathcal{Z}(P) $ the set of zeros of $ P $, we define the algebraic variety of $ \mathcal{M}(n) $ by
 $$ \mathcal{V} \equiv \mathcal{V} (\mathcal{M}(n)): = \bigcap_{P \in \mathcal{P}_{n}} \mathcal{Z}(P).$$

The following two results will be useful to explicit the representing measure of $ \beta = \beta^{(2n)} $ when it exists.
\begin{proposition}\emph{(\cite[Proposition 3.1]{CuFi1})}.\label{MyproLabe2-3} Suppose that $ \mu $ is a representing measure of $ \beta $. For $ P \in \mathcal{P}_{n} $, we have
 $$\operatorname{supp}\mu \subseteq \mathcal{Z}(P)\Longleftrightarrow P(X,Y)=\mathbf{0}.$$
\end{proposition}
Using this proposition and by virtue of Corollary 3.7 in \cite{CuFi1}, we deduce
\begin{equation*}
\operatorname{supp}\mu\subseteq \mathcal{V}(M(n))\ \text{ and }\ \operatorname{rank}\mathcal{M}(n)\leq \operatorname{card}\operatorname{supp}\mu\leq v:=\operatorname{card}\mathcal{V}.
\end{equation*}
\begin{theorem}\label{MythmLabe2-4}\emph{(\cite[Theorem 5.13]{CuFi1})}. The truncated moment sequence $ \beta^{(2n)} $ has a $ \operatorname{rank}M(n)$-atomic representing measure  if and only if $ \mathcal{M}(n) \succeq 0 $ and $ \mathcal{M}(n) $ admits a flat extension $ \mathcal{M}(n + 1) $.
\end{theorem}

If $ \mathcal{M}(n) $ admits a positive semidefinite extension $ \mathcal{M}(n + 1) $ such that $ \mathcal{M}(n + 1) $ is flat or has a flat extension  $ \mathcal{M}(n + 2) $, then $ \beta $ admits a representing measure $ \mu $ which is $r$-atomic where $ r = \operatorname{rank} \mathcal{M}(n + 1) $. By virtue of the flat extension Theorem \ref{MythmLabe2-4}, the algebraic variety $ \mathcal{V} $ of $ \mathcal{M}(n + 1) $ consists of exactly $ r $ points.

Let us put $ \mathcal{V} = \{(x_{1}, y_{1}), (x_{2}, y_{2}), \cdots, (x_{r}, y_{r}) \} $ and consider the Vandermonde matrix $ V $ given by
\begin{equation*}
V=\left(\begin{array}{ccccccc}
1 & 1 & 1 & \ldots & 1 & 1\\
x_{1} & x_{2} & x_{3} & \ldots & x_{r-1} & x_{r}\\
y_{1} & y_{2} & y_{3} & \ldots & y_{r-1} & y_{r}\\
x_{1}^2 & x_{2}^2 & x_{3}^2 & \ldots & x_{r-1}^2 & x_{r}^2\\
x_{1}y_{1} & x_{2} y_{2} & x_{3}y_{3} & \ldots & x_{r-1}y_{r-1} & x_{r}y_{r}\\
\vdots & \vdots & \vdots & \vdots & \vdots &\vdots\\
x_{1}^{n+1} & x_{2}^{n+1} & x_{3}^{n+1} & \ldots & x_{r-1}^{n+1} & x_{r}^{n+1}\\
\vdots & \vdots & \vdots & \vdots & \vdots &\vdots\\
y_{1}^{n+1} & y_{2}^{n+1} & y_{3}^{n+1} & \ldots & y_{r-1}^{n+1} & y_{r}^{n+1}
\end{array}\right).
\end{equation*}
If we denote by $ \mathcal{B} = \{c_{1}, c_{2}, \cdots, c_{r} \} $ the basis of $ \mathcal{C}_{\mathcal{M}(m)} $, the column space of $ M(n+1) $, and if $ V_{\vert \mathcal{B}} $ is the compression of $ V $ to the columns of $ \mathcal{B} $, then we can determine the weights $ \rho_{k} $ of the atoms $ \{(x_{k}, y_{k}) \} $;  $ (1\leq k \leq r) $ by solving the following Vandermonde system,
\begin{equation}
V_{\vert \mathcal{B}}(\rho_{1}\quad \rho_{2}\quad \cdots \quad \rho_{r})^{T}=(L_{\beta}(c_{1})\quad L_{\beta}(c_{2})\quad \cdots \quad L_{\beta}(c_{r}))^{T}.\label{MyEqLabe2-5}
\end{equation}
Hence, the representing measure of $\beta$ is $\mu=\sum\limits_{k=1}^{r} \rho_{k} \delta_{\left(x_{k}, y_{k}\right)}$.

We end this section with a reminder of recursively determined positive semidefinite moment matrices.\\ We denote by $ \mathcal{C}_{\mathcal{M}(n)} = \operatorname{span} \left \{1, X, Y, X^{2}, XY, Y^{2}, \cdots , X^{n}, \cdots, Y^{n} \right \} $ the column space of the matrix $ \mathcal{M}(n) $. We express the $ \mathcal{M}(n)$ columns linear dependence by the  following relations,
$$
P_{1}(X, Y)=\mathbf{0}, P_{2}(X, Y)=\mathbf{0}, \ldots, P_{k}(X, Y)=\mathbf{0},
$$
for some polynomials $P_{1}, P_{2}, \ldots, P_{k} \in \mathcal{P}_{n}, k \in \mathbb{N}$ and $k \leq \dfrac{(n+2)(n+1)}{2}$.

We recall that $ \mathcal{M}(n) $ is recursively generated \cite{FiaNi} if the following property is verified
\begin{equation}
P, Q, P Q \in \mathcal{P}_{n}, P(X, Y)=\mathbf{0} \Longrightarrow(P Q)(X, Y)=\mathbf{0}.\label{MyEqLabe2-6}
\end{equation}

According to \cite[Proposition 4.2]{Fia}, $ \mathcal{M}(n) $ is recursively determined  if it has the following column dependence relations,
\begin{align}
& X^{n}=P(X, Y)=\sum_{i+j \leq n-1} a_{i j} X^{i} Y^{j},\label{MyEqLabe2-7}\\
& Y^{n}=Q(X, Y)=\sum_{i+j \leq n, j \neq n} b_{i j} X^{i} Y^{j},\label{MyEqLabe2-8}
\end{align}
or by similar relations with reversing the roles of $ P $ and $ Q $.

In our approach for the cubic TRMP case, the following lemma will be very useful. 
\begin{lemma} \emph{(\cite[Lemma 2.4]{CuYo4})} \label{MylemLabe2-5} If $ \mathcal{M}(2) \succeq 0 $ and recursively determined (the relations \eqref{MyEqLabe2-7} and \eqref{MyEqLabe2-8} are verified with $ n = 2 $), then $ \mathcal{M}(2) $ admits a flat extension $ \mathcal{M}(3) $.
\end{lemma}
Now, we are in a position to state our main results.

\section{Statement of findings}\label{MysecLabe3}

Let $ \beta = \beta^{(3)} \equiv \left \{\beta_ {ij} \right \}_{i + j \leq 3} $ be a real doubly indexed finite sequence  with $ \beta_{00}>0$. As mentioned in Section \ref{MysecLabe1}, we can not group all the data of the sequence $ \beta $ in a single square matrix, so we have distributed the elements of the sequence over two matrices $ \mathcal{M}(1)$ and $ B(2)$ (see \eqref{MyEqLabe1-2}). Thus, to solve the problem, we have to look first for a positive semidefinite extension $ \mathcal{M}(2)$ in \eqref{MyEqLabe1-3} of $ \mathcal{M}(1) $, and then test its flatness. 

This extension takes the form $ \mathcal{M}(2) = \left(\begin{array} {cc} \mathcal{M}(1) & B (2) \\ B (2)^{T } & C( 2) \end{array} \right) $ with $ C (2) $  a Hankel block containing the quartic moments,
\begin{equation}
C(2)=\left(\begin{array}{ccc}
\beta_{40} & \beta_{31} & \beta_{22} \\
\beta_{31} & \beta_{22} & \beta_{13}  \\
\beta_{22} & \beta_{13} & \beta_{04}  
\end{array}\right).\label{MyEqLabe3-1}
\end{equation}
and $ \operatorname{Ran}B(2) \subseteq \operatorname{Ran} \mathcal{M}(1)$, i.e. there exists a matrix $ W $ such that $ \mathcal{M}(1)W=B(2)$ according to Douglas factorization lemma (see (i) of Remark \ref{MyremLabe2-2}).

As $ \mathcal{M}(1) $ is symmetric then $W^{T} \mathcal{M}(1) W$ does also. 

So, we can write
\begin{equation}
W^{T} \mathcal{M}(1) W=\left(\begin{array}{ccc}
x & a & b \\
a & y & t  \\
b & t & z  
\end{array}\right),\label{MyEqLabe3-2}
\end{equation}
where $ a, b, t, x, y $ and $ z $ are real numbers. 

According to the Theorem \ref{MythmLabe2-1}, $ \mathcal{M}(2) \succeq 0 $,  is equivalent to get the next three conditions
\begin{gather}
(i). \  \mathcal{M}(1)\succeq 0,\quad (ii). \  \mathcal{M}(1)W=B(2)\quad \text{ and }\quad (iii). \  C (2)-W^{T} \mathcal{M}(1) W \succeq 0.\label{MyEqLabe3-3}
\end{gather} 

If the extension $ \mathcal{M}(2) $ is flat, then there exists a representing measure; otherwise, we try to construct a  flat extension $ \mathcal{M}(3) $ of $ \mathcal{ M}(2) $.

In this context and before stating our main results we need the following two lemmas. Let $C(2)$ and $W^{T}\mathcal{M}(1)W$ be as defined  in \eqref{MyEqLabe3-1} and \eqref{MyEqLabe3-2} respectively and which satisfy condition $(iii)$ of \eqref{MyEqLabe3-3}. 

\begin{lemma}\label{MylemLabe3-1}

The next equivalent holds, 
$$\operatorname{rank}(C(2)-W^{T}\mathcal{M}(1)W)=0 \text{ if and only if } y=b.$$
\end{lemma}
\begin{proof}
If $\operatorname{rank}(C(2)-W^{T}\mathcal{M}(1)W)=0$ then $C(2)=W^{T}\mathcal{M}(1)W$.

Consequently, $\beta_{40}=x, \beta_{31}=a, \beta_{13}=c, \beta_{04}=z$ and $\beta_{22}=b=y$.

Conversely, if $ y = b $ then $ W^{T} \mathcal{M}(1) W $ is a Hankel matrix from which we take $ C(2)= W^{T} \mathcal{M} (1) W $.

Therefore, $ \operatorname{rank}(C (2) -W^{T} \mathcal{M}(1) W) = 0 $ and $ C (2) -W^{T} \mathcal{M}(1) W \succeq 0 $.
\end{proof}
From the Lemma \ref{MylemLabe3-1}, we deduce the following result.
\begin{lemma}\label{MylemLabe3-2}
$\operatorname{rank}(C(2)-W^{T}\mathcal{M}(1)W)\geq1$  if and only if $y\neq b$.
\end{lemma}
Now, we are in  a position to state our first result.
\begin{theorem}\label{MythmLabe3-3}
Let $ \beta = \beta^{(3)} $ be a real doubly indexed finite sequence, $ b $ and $ y $ as in \eqref{MyEqLabe3-3}. If $ \mathcal{M} (1) \succeq 0 $ and $ \operatorname{Ran}B(2)\subseteq \operatorname{Ran}\mathcal{M}(1) $ and $ b = y $, then $ \beta $ admits a unique representing measure $ \operatorname{rank} \mathcal{M}(1)$-atomic.
\end{theorem}
\begin{proof}
If $b = y$ then by Lemma \ref{MylemLabe3-1}, we have $ C (2)= W^{T} \mathcal{M}(1) W $.

Therefore, $ \mathcal{M}(2) $ is a flat extension of $ \mathcal{M}(1) $.

Consequently it is positive semidefinite and recursively determined.

Hence, by applying Lemma \ref{MylemLabe2-5}, $ \mathcal{M}(2) $ admits a flat extension $ \mathcal{M}(3) $, therefore $\beta^{(4)}$, and particularly $\beta^{(3)}$ admits a unique representing measure $\operatorname{rank} \mathcal{M}(1)$-atomic. The uniqueness comes from that of $ C(2)$.
\end{proof}
Before giving our second result, we need the following proposition.
\begin{proposition}\label{MyproLabe3-4}
If $ \mathcal{M}(1) \succeq 0 $ with $r=\operatorname{rank}\mathcal{M}(1)=1 \text{ or } 2$ and $\operatorname{Ran}B(2)~\subseteq ~\operatorname{Ran}M(1)$ then $y=b$.
\end{proposition}
\begin{proof} For both cases $ r=1$ or $ r=2$, we assume that $ \beta_{00} = 1$.
\begin{itemize}
\item[(i)] If $ r = 1$, and without loss of generality, we suppose that
\begin{gather*}
\mathcal{M}(1)=\left(\begin{array}{ccc}
1 & c & d  \\
c & c^2 & cd\\
d & cd & d^2 
\end{array}\right) \text{ and } B(2)=\left(\begin{array}{ccc}
c^2 & cd & d^2  \\
c^3 & c^2d & cd^2\\
c^2d & cd^2 & d^3 
\end{array}\right),
\end{gather*}
where $c$ and $d$ are both non-zero real numbers, since $c=d=0$ is a trivial case.

It is easy to check that $\mathcal{M}(1)\succeq 0$, $r=\operatorname{rank}\mathcal{M}(1)=1$ and that the linear dependency relations between the columns of $\mathcal{M}(1)$ are $X=c.1$ and $Y=d.1$.

A calculation shows us that $\mathcal{M}(1)W=B(2)$, with $ 
W=\left(\begin{array}{ccc}
c^2 & cd & d^2  \\
0 & 0 & 0\\
0 & 0 & 0 
\end{array}\right)
$, i.e., $\operatorname{Ran}B(2)\subseteq \operatorname{Ran}M(1)$ therefore
\begin{gather*}
W^{T}\mathcal{M}(1)W=\left(\begin{array}{ccc}
c^4 & c^3d & c^2d^2  \\
c^3d & c^2d^2 & cd^3\\
c^2d^2 & cd^3 & d^4 
\end{array}\right).
\end{gather*}
Hence, $y=b=c^2d^2$
\item[(ii)]  If $r=2$, there are four possibilities for the linear dependency relations between the columns of $\mathcal{M}(1)$, that we can group into two: $X=c.1$ or $Y=d.1+e.X$.

\item If $X=c.1$, then without loss of generality we assume that
\begin{gather*}
\mathcal{M}(1)=\left(\begin{array}{ccc}
1 & c & d  \\
c & c^2 & cd\\
d & cd & e 
\end{array}\right) \text{ and } B(2)=\left(\begin{array}{ccc}
c^2 & cd & e \\
c^3 & c^2d & ce\\
c^2d & ce & ed 
\end{array}\right),
\end{gather*}
 where $c, d, \in \mathbb{R}$ and $e>d^2$.

We check that $\mathcal{M}(1)\succeq 0$, $r=\operatorname{rank}\mathcal{M}(1)=2$ and that the linear dependence relation between the columns of $\mathcal{M}(1)$ is $X=c.1$.

A calculation shows that $\operatorname{Ran}B(2)\subseteq \operatorname{Ran}\mathcal{M}(1)$ with $W=\left(
\begin{array}{ccc}
 c^2 & 0 & e \\
 0 & 0 & 0 \\
 0 & c & 0 \\
\end{array}
\right).$
So
$$
W^{T}\mathcal{M}(1)W=\left(
\begin{array}{ccc}
 c^4 & c^3 d & c^2 e \\
 c^3 d & c^2 e & c d e \\
 c^2 e & c d e & e^2 \\
\end{array}
\right).$$
Hence, $y=b=ec^2$.

\item If $Y=d.1+e.X$ with $d$ and $e$ be non-zero real numbers, then without loss of generality we set 
$$\mathcal{M}(1)=\left(
\begin{array}{ccc}
 1 & c & c e+d \\
 c & f & c d+e f \\
 c e+d & c d+e f & d^2+e^2f+2 cde \\
\end{array}
\right), \text{ where } f, c\in \mathbb{R} \text{ and } f>c^2,$$
and
$$
B(2)=\left(
\begin{array}{ccc}
 f & c d+e f & d^2+e^2f+2 cde \\
 c f & c e f+d f & c d^2+c e^2 f+2 d e f \\
 c e f+d f & c d^2+c e^2 f+2 d e f & d^3+c e^3 f + 3 c d^2 e+3 d e^2 f \\
\end{array}
\right).$$
A simple check shows us that $\mathcal{M}(1)\succeq 0$, $r=\operatorname{rank}\mathcal{M}(1)=2$ and that the linear dependence relation between the columns of $\mathcal{M}(1)$ is $Y=d.1+e.X$. 

A calculation shows us that $\operatorname{Ran}B(2)\subseteq \operatorname{Ran}\mathcal{M}(1)$ with
$W=\left(
\begin{array}{ccc}
 f & e f & d^2+e^2 f \\
 0 & d & 2 d e \\
 0 & 0 & 0 \\
\end{array}
\right)$, and $W^{T}\mathcal{M}W$ as follows
$$
\left(
\begin{array}{ccc}
 f^2 & c d f+e f^2 & 2 c d e f+d^2 f+e^2 f^2 \\
 c d f+e f^2 & 2 c d e f+d^2 f+e^2 f^2 & c d^3+3 c d e^2 f+3 d^2 e f+e^3 f^2 \\
 2 c d e f+d^2 f^2+e^2 f^2 & c d^3+3 c d e^2 f+3 d^2 e f+e^3 f^2 & 4 e c d^3+4 e^3 c d f+6 e^2 d^2 f+d^4+e^4 f^2 \\
\end{array}
\right).$$
Hence $y=b=2 c d e f+d^2 f^2+e^2 f^2 $
\end{itemize}
The proof is thus completed. 
 \end{proof}
Applying Proposition \ref{MyproLabe3-4} and the Theorem \ref{MythmLabe3-3}, we establish directly the following theorem.
\begin{theorem}\label{MythmLabe3-5}
Let $ \beta = \beta^{(3)}$  be a real doubly indexed finite sequence. If $ \mathcal{M}(1) \succeq 0 $ and $ r = \operatorname{rank}\mathcal{M}(1) = 1 \text{ or } 2 $, with $ \operatorname{Ran}B(2) \subseteq \operatorname{Ran}M (1) $, then $ \beta $ admits a unique representing measure $ \operatorname{rank} \mathcal{M}(1)$-atomic.
\end{theorem}
To highlight this last theorem, we present the following two numerical examples.
\begin{example} \label{MyexLabe3-6} The case where $ \mathcal{M}(1) $ is singular. \\
Let $ \beta^{(3)} $ be a real doubly indexed finite sequence with $ \beta_{00} = 5 $, $ \beta_{10} = 1 $, $ \beta_{01} = 2 $, $ \beta_{20} = 5$, $ \beta_{11} = -2 $, $ \beta_{02} = 2$, $ \beta_{30}= 1$, $ \beta_{21}=2 $ , $  \beta_{12} = - 2$ and $ \beta_{03} = 2$ . \\
The two matrices associated to $ \beta^{(3)}$ are,
$$
\mathcal{M}(1)=\left(
\begin{array}{ccc}
 5 & 1 & 2 \\
 1 & 5 & -2 \\
 2 & -2 & 2 \\
\end{array}
\right) \text{ and }B(2)=\left(
\begin{array}{ccc}
 5 & -2 & 2 \\
 1 & 2 & -2 \\
 2 & -2 & 2 \\
\end{array}
\right).
$$
Calculations, with Mathematica software show that $\mathcal{M}(1)\succeq 0$ and $\operatorname{rank}\mathcal{M}(1)=~2$.\\ 
$W$ and $W^{T}\mathcal{M}(1)W$ are given by,
$$
W=\left(
\begin{array}{ccc}
 1 & -\frac{1}{2} & \frac{1}{2} \\
 0 & \frac{1}{2} & -\frac{1}{2} \\
 0 & 0 & 0 \\
\end{array}
\right)\quad \text{ and } \quad W^{T}\mathcal{M}(1)W=\left(
\begin{array}{ccc}
 5 & -2 & 2 \\
 -2 & 2 & -2 \\
 2 & -2 & 2 \\
\end{array}
\right).
$$
We notice that $ b = y =2 $, so according to Theorem \ref{MythmLabe3-3}, $ \beta^{(3)}$ admits a unique representing measure $2$-atomic. By  choosing $ C (2) = W^{T} \mathcal{M}(1) W $, the matrix $\mathcal{M}(1)$ admits a flat  extension $ \mathcal{M}(2) $  ($ \operatorname{rank} \mathcal{M} (2 ) = \operatorname{rank} \mathcal{M} (1) = 2 $) given by
$$
 \mathcal{M}(2)=\left(
\begin{array}{cccccc}
 5 & 1 & 2 & 5 & -2 & 2 \\
 1 & 5 & -2 & 1 & 2 & -2 \\
 2 & -2 & 2 & 2 & -2 & 2 \\
 5 & 1 & 2 & 5 & -2 & 2 \\
 -2 & 2 & -2 & -2 & 2 & -2 \\
 2 & -2 & 2 & 2 & -2 & 2 \\
\end{array}
\right).$$ 
$ \mathcal{M}(2) $ columns dependence relations are,
$$ X + 2Y-1 = 0, X ^ 2-1 = 0, -X + 2XY + 1 = 0 \text { and } 2Y ^ 2 + X-1 = 0. $$
Thus, the algebraic variety of $ \mathcal{M}(2) $ is $ \mathcal{V} = \{(1,0); (- 1; 1) \} $, and by solving the Vandermonde system \eqref{MyEqLabe2-5}, we find the weights $ \rho_{1} = 3 $ and $ \rho_{2} = 2 $ related to the atoms $ (1,0) $ and $ (- 1; 1) $ respectively. Finally, the representing measure $2$-atomic of $ \beta^{(3)} $ is
 $$ \mu = 3 \delta_{(1,0)} + 2 \delta_{(- 1,1)}. $$
\end{example}
\begin{example}\label{MyexLabe3-7} The case where $ \mathcal{M}(1) $ is nonsingular. \\
Let $ \beta^{(3)} $ be a real doubly indexed finite sequence with $ \beta_{00}= 3 $, $  \beta_{10}=2 $, $ \beta_{01}=2$, $ \beta_{20}2=$, $ \beta_{11}=-1$, $\beta_{02}=2$, $\beta_{30}=2$, $\beta_{21}=-1$, $ \beta_{12}=1$ and $ \beta_{03}=0$. \\
The two matrices associated to $ \beta^{(3)} $ are,
$$
\mathcal{M}(1)=\left(
\begin{array}{ccc}
 3 & 2 & 0 \\
 2 & 2 & -1 \\
 0 & -1 & 2 \\
\end{array}
\right)\text{ and } B(2)=\left(
\begin{array}{ccc}
 2 & -1 & 2 \\
 2 & -1 & 1 \\
 -1 & 1 & 0 \\
\end{array}
\right).
$$
A computation of the nested determinants of $ \mathcal{M} (1) $ shows  that $ \mathcal{M}(1)> 0. $
 
Therefore, $ \operatorname {rank}\mathcal{M}(1)=3$ and
$$
 W=(\mathcal{M}(1))^{-1}B(2)=\left(
\begin{array}{ccc}
 0 & -1 & 2 \\
 1 & 1 & -2 \\
 0 & 1 & -1 \\
\end{array}
\right)\text{ and } W^{T}\mathcal{M}(1)W=\left(
\begin{array}{ccc}
 2 & -1 & 1 \\
 -1 & 1 & -1 \\
 1 & -1 & 2 \\
\end{array}
\right).
$$
Since $b=y=1$, then from Theorem \ref{MythmLabe3-3}, we deduce that $\beta^{(3)}$ admits a unique representing measure $3$-atomic and $ \mathcal{M}(1) $ admits a flat extension $ \mathcal{M}(2)$  ($ \operatorname{rank} \mathcal{M}(2 ) = \operatorname{rank} \mathcal{M}(1) = 3$).

By choosing $ C (2) = W ^ {T} \mathcal{M}(1) W$, we get
$$
 \mathcal{M}(2)=\left(
\begin{array}{cccccc}
 3 & 2 & 0 & 2 & -1 & 2 \\
 2 & 2 & -1 & 2 & -1 & 1 \\
 0 & -1 & 2 & -1 & 1 & 0 \\
 2 & 2 & -1 & 2 & -1 & 1 \\
 -1 & -1 & 1 & -1 & 1 & -1 \\
 2 & 1 & 0 & 1 & -1 & 2 \\
\end{array}
\right).$$
The $ \mathcal{M}(2) $ columns dependence relations are,
$$ X^2-X=0, Y^2+Y+2X-2=0 \text{ and } XY-Y-X+1=0.$$
Thus, the algebraic variety of $ \mathcal{M}(2) $ is $ \mathcal{V} = \{(0,1); (1; -1); (1; 0) \} $, and by solving the Vandermonde system \eqref{MyEqLabe2-5} we find the weights $ \rho_{1} = \rho_{2} = \rho_{3} = 1$.

Finally, the representing measure $ 3$-atomic of $ \beta^{(3)} $ is,
 $$ \mu = \delta_{(0,1)} + \delta_{(1, -1)} + \delta_{(1.0)}. $$
\end{example}

Let us now give our last result concerning the case $ \mathcal{M}(1)> 0 $ and $ y \neq b $,
\begin{theorem}\label{MythmLabe3-8}
Let $ \beta = \beta^{(3)} $ be a real doubly indexed finite sequence, $ b $ and $ y $ be defined as in \eqref{MyEqLabe3-2}. If $ \mathcal{M}(1)> 0 $, $ \operatorname{Ran}B(2) \subseteq \operatorname{Ran}\mathcal{M}(1) $ and $ b \neq y $, then $ \beta $ admits a representing measure $4$-atomic.
\end{theorem}
\begin{proof}
Since $b\neq y$ then for appropriate quartic moments (the entries of block $C(2)$), and according to Lemma \ref{MylemLabe3-2}, we must have $ \operatorname{rank}(C(2)-W^{T}\mathcal{M}(1)W)\geq 1$.

If $ b> y $, with the quartic moments given by
\begin{gather}
\beta_{40}= x, \ \beta_{31} = a, \ \beta_{22} = b, \ \beta_{13} = c  \ \text{and} \ \beta_{04} = z, \label{MyEqLabe3-4}
\end{gather}
we have,
$ C (2)-W^{T} \mathcal{M}(1)W = \left(\begin{array}{ccc}
0 & 0 & 0 \\
0 & b-y & 0 \\
0 & 0 & 0
\end{array} \right)
$
which is positive semidefinite matrix of rank 1.

If $ b <y $, by taking
\begin{gather}
\beta_{40}= x+1, \ \beta_{31} = a, \ \beta_{22} = y, \ \beta_{13} = c \ \text {and} \ \beta_{04} = (y-b)^2-z, \label{MyEqLabe3-5}
\end{gather}
we get,
$
C(2)-W^{T} \mathcal{M}(1)W = \left (\begin{array}{ccc}
1 & 0 & y-b \\
0 & 0 & 0 \\
y-b & 0 & (y-b)^2
\end{array} \right)
$
which is also a positive semidefinite matrix of rank 1.

Consequently, when $y\neq b$ we have $ \operatorname{rank}(C(2)-W^{T}\mathcal{M}(1)W)=1$. As $ C(2)-W^{T} \mathcal{M}(1)W \succeq 0 $, according to Theorem \ref{MythmLabe2-1}, the extension matrix $ \mathcal{M}(2)$ of $ \mathcal{M}(1)$ that we have built is positive semidefinite. In addition, by the relation \eqref{MyEqLabe2-4} we have $ \operatorname{rank} \mathcal{M}(2)= 4 $.

Hence, there exists a column in $ \mathcal{M}(2) $ linearly independent with the columns $ 1, X $ and $ Y $. This column is  $ X^2 $ if $ \beta_{22} = y $ or $ XY $ if $ \beta_{40}= x $. In fact, \\
if $ \beta_{40}= x $, then  $\operatorname{det}\left(\begin{array}{lc}
\mathcal{M}(1) & (X^{2})\\
(X^{2})^{T} & x
\end{array}\right)=0$
 and if $ \beta_{22}= y $, $\operatorname{det}\left(\begin{array}{lc}
\mathcal{M}(1) & (XY)\\
(XY)^{T} & y
\end{array}\right)~=~0$.
 
For the case $ b> y $, the column linearly independent with the columns $ 1, X $ and $ Y $ is $ XY $ (see \ref{MyEqLabe3-4}). 

Hence, the columns $ X^2 $ and $ Y^2 $ are
\begin{equation}
 X^2=\alpha_{1}XY+a_{0}+a_{1}X+a_{2}Y \text{ and } Y^2=\alpha_{2}XY+b_{0}+b_{1}X+b_{2}Y.\label{MyEqLabe3-6}
\end{equation}
with 
\begin{equation}
\begin{array}{ll}
\alpha_{1} 
&=\dfrac{\operatorname{det}\left(\begin{array}{lc}
\mathcal{M}(1) & X^{2}\\
(XY)^{T} & \beta_{31}
\end{array}\right)}{\operatorname{det}\left(\begin{array}{lc}
\mathcal{M}(1) & XY \\
(XY)^{T} & \beta_{22} 
\end{array}\right)}\\
&=\dfrac{\operatorname{det}\left(\begin{array}{lc}
\mathcal{M}(1) &  X^{2} \\
(XY)^{T} & a
\end{array}\right)+(\beta_{31}-a) \operatorname{det}\left(\mathcal{M}(1)\right)}{\operatorname{det}\left(\begin{array}{lc}
\mathcal{M}(1) & XY\\
(XY)^{T} & y
\end{array}\right)+(\beta_{22}-y) \operatorname{det}\left(\mathcal{M}(1)\right)}\\
&=\dfrac{(\beta_{31}-a) \operatorname{det}\left(\mathcal{M}(1)\right)}{(\beta_{22}-y) \operatorname{det}\left(\mathcal{M}(1)\right)}\\
&=\dfrac{\beta_{31}-a}{\beta_{22}-y} = 0,\qquad (\beta_{31}=a \text{ and } \beta_{22}=b> y). 
\end{array} \label{MyEqLabe3-7}
\end{equation}
Similar calculations as in \eqref{MyEqLabe3-7} give  $ \alpha_{2} = \dfrac{\beta_{13}-c}{\beta_{22}-y} = 0 \ (\beta_{13} = c \text{ and } \beta_{22} = b> y) $. 

Finally, the relations \eqref{MyEqLabe3-6} become,
\begin{equation*}
 X^2=a_{0}+a_{1}X+a_{2}Y \text{ and }  Y^2=b_{0}+b_{1}X+b_{2}Y.
\end{equation*}
Consequently, $ \mathcal{M}(2) \succeq 0 $ and recursively determined.

Therefore, according to Lemma \ref{MylemLabe2-5}, $ \mathcal{M}(2) $ admits a flat extension, where $ \beta^{(4)} $ a fortiori $ \beta^{(3)} $ admits a representing measure $4$-atomic.

For the case $ b <y $, the column $ X^2 $  is linearly independent with the columns $ 1, X $ and $ Y $ in $ \mathcal{M}(2)$ (see \ref{MyEqLabe3-4}). 

Let us take $ \beta_{ 22} = y $ and $ \beta_{40} \neq x $ \eqref{MyEqLabe3-5}.

So, the columns $ XY $ and $ Y^2$ can be written as follows,
\begin{equation}
XY= \alpha_{2}X^2+c_{0}+c_{1}X+c_{2}Y \text{ and } Y^2=\alpha_{3}X^2+d_{0}+d_{1}X+d_{2}Y. \label{MyEqLabe3-8}
\end{equation}
By calculations as in \eqref{MyEqLabe3-7}, we find 
$$\alpha_{2} =\frac{\beta_{31}-a}{\beta_{40}-x}= 0 \text{ and }  \alpha_{3} = \frac{\beta_{22}-b}{\beta_{40}-x} \neq 0, (\beta_{40}\neq x, \beta_{31}=a \text{ and }  \beta_{22}=y\neq b).$$
So the relations \eqref{MyEqLabe3-8} become as follows,
\begin{align}
& XY= c_{0}+c_{1}X+c_{2}Y,\label{MyEqLabe3-9}\\
& Y^2=(y-b)X^2+d_{0}+d_{1}X+d_{2}Y.\label{MyEqLabe3-10}
\end{align}
Now, we focus on constructing the positive semidefinite extension $ \mathcal{M}(3) $ of $ \mathcal{M}(2) $.

As the condition of the recursivily generated   must be respected, then  from the relations \eqref{MyEqLabe3-9} and \eqref{MyEqLabe3-10} and by functional calculus, we obtain
\begin{align}
& X^2Y= c_{0}X+c_{1}X^2+c_{2}XY, \label{MyEqLabe3-11}\\
& XY^2=c_{0}Y+c_{1}XY+c_{2}Y^2, \label{MyEqLabe3-12}\\
& XY^2=(y-b)X^3+d_{0}X+d_{1}X^2+d_{2}XY. \label{MyEqLabe3-13}
\end{align}
Using the relations \eqref{MyEqLabe3-9}-\eqref{MyEqLabe3-11}, we get  
\begin{multline}
Y^3=[c_{0}(c_{2}y-c_{2}b+d_{1})+d_{0}d_{2}]+[(c_{0}+c_{1}c_{2})(y-b)+c_{1}d_{1}+d_{1}d_{2}]X+\\
[c_{2}(c_{2}y-c_{2}b+d_{1})+d_{0}+d_{2}^2]Y+(c_{1}+d_{2})(y-b)X^2 \label{MyEqLabe3-14}
\end{multline}
Noticing that the column $ XY^2 $ is defined by the relations \eqref{MyEqLabe3-12} and \eqref{MyEqLabe3-13}, then by the property  \eqref{MyEqLabe2-6}, these two relations must be similar.

Furthermore, since $ y \neq b $ then 
\begin{equation}
 X^3=-\left(\frac{d_{0}}{y-b}\right)X+\left(\frac{c_{0}}{y-b}\right)Y-\left(\frac{d_{1}}{y-b}\right)X^2+\left(\frac{c_{1}-d_{2}}{y-b}\right)XY+c_{2}Y^2. \label{MyEqLabe3-15}
\end{equation}

Thus, using the definition of the columns $ X^3, X^2Y, XY^2 $ and $ Y^3 $, and by the relations \eqref{MyEqLabe3-15}, \eqref{MyEqLabe3-11}, \eqref{MyEqLabe3-12} or \eqref{MyEqLabe3-13} and \eqref{MyEqLabe3-14} respectively, we complete the construction of the matrix $ \mathcal{M}(3)$ as detailed in Remark \ref{MyremLabe3-9} below.

Finally, Since these columns are written as a linear combination of columns associated to monomials of degree at most 2, then $ \mathcal{M}(3) $ is a flat extension of $ \mathcal{M}(2)$.
 
Whence, $ \beta^{(3)} $ admits a finite measure $4$-atomic.
\end{proof}
\begin{remark}\label{MyremLabe3-9}
Practically, if $ \mathcal{M}(2) $ is recursively determined, to construct the matrix $ \mathcal{M}(3)$, we define the columns $ X^3 $ and $ Y^3 $ by the functional calculation and the definitions of the columns $ X^2 $ and $ Y^2 $. Then, we compute the quintic moments in the columns $ X^3 $ and $ Y^3 $. This allows us to build the Hankel block $ B [2,3]$. Thus, the construction of the block $ B (3) $ is completed.

 By transposing the latter, one can construct the block $ C (3) $ as previously, which completes the construction of $ \mathcal{M} (3) $. 
 
 If $ \mathcal{M} (2) $ is not recursively determined, then with the relations \eqref{MyEqLabe3-11}, \eqref{MyEqLabe3-12} or \eqref{MyEqLabe3-13}, \eqref{MyEqLabe3-14} and \eqref{MyEqLabe3-15}, we start calculating the quintic moments without conflict in order to complete the construction of the block $ B (3) $, then we transpose $ B (3) $ to calculate $ C (3) $, which contains the sixth moments.
\end{remark}

Now, we give two numerical examples illustrating both cases in Theorem~\ref{MythmLabe3-8}.
\begin{example}\label{MyexLabe3-10} Case where $b>y$

Let $ \beta^{(3)} $ be the be a real doubly indexed finite sequence  defined by
$\beta_{00}=2$, $\beta_{10}=1$, $\beta_{01}=1$, $\beta_{20}=2$, $\beta_{11}=1$, $\beta_{02}=2$,  $\beta_{30}=1$, $\beta_{21}=2$, $\beta_{12}=1$ and $\beta_{03}=2$.

The two matrices associated to $ \beta^{(3)} $ are,
$$
\mathcal{M}(1)=\left(
\begin{array}{ccc}
 2 & 1 & 1 \\
 1 & 2 & 1 \\
 1 & 1 & 2 
\end{array}\right)
\text{ and } B(2)=\left(
\begin{array}{ccc}
 2 & 1 & 2 \\
 1 & 2 & 1 \\
 2 & 1 & 2 
\end{array}
\right).$$
Calculations show that $ \mathcal{M}(1)> 0$ and $ \operatorname{rank} \mathcal{M}(1)=3$.
 
So,
$$
W=\mathcal{M}(1)^{-1}B(2)=\left(
\begin{array}{ccc}
 \frac{3}{4} & 0 & \frac{3}{4} \\
 \frac{-1}{4} & 1 & \frac{-1}{4} \\
 \frac{3}{4} & 0 & \frac{3}{4} \\
\end{array}
\right)\text{ and } W^{T}\mathcal{M}(1)W=\left(
\begin{array}{ccc}
 \frac{11}{4} & 1 & \frac{11}{4}\\
 1 & 2 & 1 \\
 \frac{11}{4} & 1 & \frac{11}{4}\\
\end{array}
\right).$$
We have  $b= \frac{11}{4}>y=2$.

 By the relation \eqref{MyEqLabe3-4}, we set $ C (2)=\left(
\begin{array}{ccc}
 \frac{11}{4} & 1 & \frac{11}{4} \\
 1 & \frac{11}{4} & 1 \\
 \frac{11}{4} & 1 & \frac{11}{4} \\
\end{array}
\right)$.

 Then the extension $ \mathcal{M}(2) $ of $ \mathcal{M}(1) $ is,
$$
\mathcal{M}(2)=\left(
\begin{array}{cccccc}
 2 & 1 & 1 & 2 & 1 & 2 \\
 1 & 2 & 1 & 1 & 2 & 1 \\
 1 & 1 & 2 & 2 & 1 & 2 \\
 2 & 1 & 2 & \frac{11}{4} & 1 & \frac{11}{4} \\
 1 & 2 & 1 & 1 & \frac{11}{4} & 1 \\
 2 & 1 & 2 & \frac{11}{4} & 1 & \frac{11}{4} \\
\end{array}
\right).$$
The  computation of the nested determinants shows that $ \mathcal{M}(2) \succeq 0 $ and the dependency relations between the columns are,
$$ X^2=\frac{3}{4}-\frac{1}{4}X+\frac{3}{4}Y\text{ and } Y^2=\frac{3}{4}-\frac{1}{4}X+\frac{3}{4}Y. $$
Further, the algebraic variety of $ \mathcal{M}(2) $ is,
 $$\mathcal{V}=\left\{\left(\frac{-3}{2},\frac{3}{2}\right);\left(\frac{1}{2},\frac{-1}{2}\right);\left(\frac{1-\sqrt{13}}{4},\frac{1-\sqrt{13}}{4}\right);\left(\frac{1+\sqrt{13}}{4},\frac{1+\sqrt{13}}{4}\right)\right\}.$$
With solving the Vandermonde system \eqref{MyEqLabe2-5}, we get the weights 
$$\rho_{1}=\frac{1}{6}, \rho_{2}=\frac{1}{2}, \rho_{3}=\frac{2}{39} \left(13-2 \sqrt{13}\right)\text{ and } \rho_{4}=\frac{2}{39} \left(2 \sqrt{13}+13\right),$$ 
related respectively to the following atoms 
$$\left(\frac{-3}{2},\frac{3}{2}\right), \left(\frac{1}{2},\frac{-1}{2}\right),  \left(\frac{1-\sqrt{13}}{4},\frac{1-\sqrt{13}}{4}\right)\text{ and } \left(\frac{1+\sqrt{13}}{4},\frac{1+\sqrt{13}}{4}\right).$$ 

Finally the $4$-atomic measure of $ \beta^{(3)} $ is,
$$\mu=\frac{1}{6}\delta_{\left(\frac{-3}{2},\frac{3}{2}\right)}+\frac{1}{2}\delta_{\left(\frac{1}{2},\frac{-1}{2}\right)}+ \frac{26-4 \sqrt{13}}{39}\delta_{\left(\frac{1-\sqrt{13}}{4},\frac{1-\sqrt{13}}{4}\right)}+\frac{26+4 \sqrt{13}}{39}\delta_{\left(\frac{1+\sqrt{13}}{4},\frac{1+\sqrt{13}}{4}\right)}.$$
Using the technique described in Remark \ref{MyremLabe3-9}, we construct $\mathcal{M}(3)$ and we obtain
\renewcommand{\arraystretch}{1.5}
$$
\mathcal{M}(3)=\left(
\begin{array}{cccccccccc}
 2 & 1 & 1 & 2 & 1 & 2 & 1 & 2 & 1 & 2 \\
 1 & 2 & 1 & 1 & 2 & 1 & \frac{11}{4} & 1 & \frac{11}{4} & 1 \\
 1 & 1 & 2 & 2 & 1 & 2 & 1 & \frac{11}{4} & 1 & \frac{11}{4} \\
 2 & 1 & 2 & \frac{11}{4} & 1 & \frac{11}{4} & \frac{13}{16} & \frac{53}{16} & \frac{13}{16} & \frac{53}{16} \\
 1 & 2 & 1 & 1 & \frac{11}{4} & 1 & \frac{53}{16} & \frac{13}{16} & \frac{53}{16} & \frac{13}{16} \\
 2 & 1 & 2 & \frac{11}{4} & 1 & \frac{11}{4} & \frac{13}{16} & \frac{53}{16} & \frac{13}{16} & \frac{53}{16} \\
 1 & \frac{11}{4} & 1 & \frac{13}{16} & \frac{53}{16} & \frac{13}{16} & \frac{139}{32} & \frac{17}{32} & \frac{139}{32} & \frac{17}{32} \\
 2 & 1 & \frac{11}{4} & \frac{53}{16} & \frac{13}{16} & \frac{53}{16} & \frac{17}{32} & \frac{139}{32} & \frac{17}{32} & \frac{139}{32} \\
 1 & \frac{11}{4} & 1 & \frac{13}{16} & \frac{53}{16} & \frac{13}{16} & \frac{139}{32} & \frac{17}{32} & \frac{139}{32} & \frac{17}{32} \\
 2 & 1 & \frac{11}{4} & \frac{53}{16} & \frac{13}{16} & \frac{53}{16} & \frac{17}{32} & \frac{139}{32} & \frac{17}{32} & \frac{139}{32} \\
\end{array}
\right).
$$
Computation shows that $ \operatorname{rank} \mathcal{M}(3) = \operatorname{rank} \mathcal{M}(2)= 4 $. Consequently, $ \mathcal{M}(3) $ is a flat extension of $ \mathcal{M}(2) $.
\end{example}
\begin{example}\label{MyexLabe3-11} Case where $ b <y $.

Let $ \beta^{(3)} $ be the a real doubly indexed finite sequence  defined by
$\beta_{00}=3$, $\beta_{10}=3$, $\beta_{01}=1$, $\beta_{20}=5$, $\beta_{11}=-3$, $\beta_{02}=9$,  $\beta_{30}=9$, $\beta_{21}=3$, $\beta_{12}=1$ and $\beta_{03}=1$. The two matrices associated to $ \beta^{(3)} $ are,
$$
\mathcal{M}(1)=\left(
\begin{array}{ccc}
 3 & 1 & 1 \\
 1 & 5 & -3 \\
 1 & -3 & 9 \\
\end{array}\right)\text{ and } B(2)=\left(
\begin{array}{ccc}
 5 & -3 & 9 \\
 9 & 3 & 1 \\
 3 & 1 & 1 \\
\end{array}\right).$$
Calculations show that $ \mathcal{M}(1)> 0 $  and $ \operatorname{rank}\mathcal{M}(1)= 3 $.

So,
\renewcommand{\arraystretch}{1.3}
$$ W = \mathcal{M}(1)^{-1} B (2) = \left(
\begin{array}{ccc}
 \frac{6}{11} & -\frac{19}{11} & \frac{38}{11} \\
 \frac{51}{22} & \frac{31}{22} & -\frac{9}{11} \\
 \frac{23}{22} & \frac{17}{22} & -\frac{6}{11} \\
\end{array}
\right) \text{ and } W^{T}\mathcal{M}(1)W=\left(
\begin{array}{ccc}
 \frac{294}{11} & \frac{70}{11} & \frac{91}{11} \\
 \frac{70}{11} & \frac{112}{11} & -\frac{147}{11} \\
 \frac{91}{11} & -\frac{147}{11} & \frac{327}{11} \\
\end{array}
\right).$$
We have, $b= \frac{91}{11}<y= \frac{112}{11}$,

 So according to the relation \eqref{MyEqLabe3-5}, we set $ C (2)=\left(
\begin{array}{ccc}
 \frac{305}{11} & \frac{70}{11} & \frac{112}{11} \\
 \frac{70}{11} & \frac{112}{11} & -\frac{147}{11} \\
 \frac{112}{11} & -\frac{147}{11} & \frac{4038}{121} \\
\end{array}
\right)$.

 Then the extension $ \mathcal{M}(2) $ of $ \mathcal{M}(1)$ is,
$$
\mathcal{M}(2)=\left(
\begin{array}{cccccc}
 3 & 1 & 1 & 5 & -3 & 9 \\
 1 & 5 & -3 & 9 & 3 & 1 \\
 1 & -3 & 9 & 3 & 1 & 1 \\
 5 & 9 & 3 & \frac{305}{11} & \frac{70}{11} & \frac{112}{11} \\
 -3 & 3 & 1 & \frac{70}{11} & \frac{112}{11} & -\frac{147}{11} \\
 9 & 1 & 1 & \frac{112}{11} & -\frac{147}{11} & \frac{4038}{121} \\
\end{array}
\right).$$
One can easily check that $ \mathcal{M}(2) \succeq 0 $ and the dependency relations between the columns are,
\begin{gather}
 XY= -\frac{19}{11}+\frac{31}{22}X+\frac{17}{22}Y \text{ and } Y^2=\frac{292}{121}-\frac{1269}{242}X-\frac{615}{242}Y+\frac{21}{11}X^2.\label{MyEqLabe3-16}
\end{gather}

Also the algebraic variety of $\mathcal{M}(2)$ is $\mathcal{V}=\{(x_i,y_i)\}_{i=1}^{i=4}$ where\\ $(x_1,y_1)\approx (-0,8078;1,813)$, $(x_2,y_2)~\approx~ (0,9523;-2,1455)$, $(x_3,y_3)\approx (1,1039;-0,5186)$ and $(x_4,y_4)\approx (3,0437;1,128)$.

Solving the Vandermonde system \eqref{MyEqLabe2-5}, we obtain the weights $ \rho_{1} \approx 1.44984$, $ \rho_{2} \approx 0.811033$, $ \rho_{3} \approx 0,438617 $ and $ \rho_{4} \approx  0.300505$ associated  to the atoms mentioned above respectively.

 Finally, the $4$-atomic measure of $ \beta^{(3)} $ is $ \mu = \sum\limits_{i = 1}^{4} \rho_{i} \delta_{\left(x_{i }, y_{i} \right)} $.
\end{example}

The functional calculation on the dependency relations between the columns \eqref{MyEqLabe3-16}, define the columns $X^3 $, $ X^2Y $, $ XY^2 $ and $ Y^3 $ as linear dependency functions of the leftmost columns respectively,
\begin{align*}
& X^3=-\frac{292}{231}X-\frac{19}{21}Y+\frac{478}{231}XY+\frac{423}{154}X^2+\frac{17}{42}X^2,\\
& X^2Y=-\frac{19}{11}X+\frac{31}{22}X^2+\frac{17}{22}XY,\\
& XY^2=-\frac{19}{11}Y+\frac{31}{22}XY^2+\frac{17}{22}Y^2,\\
&Y^3=\frac{292}{121}Y-\frac{1269}{242}XY-\frac{615}{242}Y^2+\frac{21}{11}X^2Y.
\end{align*}
With these definitions, we construct the extension $ \mathcal{M}(3) $ of $ \mathcal{M} (2) $ as mentioned in Remark \ref{MyremLabe3-9}, is
\renewcommand{\arraystretch}{1.5}
\begin{small}
$$\left(
\begin{array}{cccccccccc}
 3 & 1 & 1 & 5 & -3 & 9 & 9 & 3 & 1 & 1 \\
 1 & 5 & -3 & 9 & 3 & 1 & \frac{305}{11} & \frac{70}{11} & \frac{112}{11} & -\frac{147}{11} \\
 1 & -3 & 9 & 3 & 1 & 1 & \frac{70}{11} & \frac{112}{11} & -\frac{147}{11} & \frac{4038}{121} \\
 5 & 9 & 3 & \frac{305}{11} & \frac{70}{11} & \frac{112}{11} & \frac{134433}{1694} & \frac{6883}{242} & \frac{1410}{121} & \frac{555}{242} \\
 -3 & 3 & 1 & \frac{70}{11} & \frac{112}{11} & -\frac{147}{11} & \frac{6883}{242} & \frac{1410}{121} & \frac{555}{242} & \frac{13921}{2662} \\
 9 & 1 & 1 & \frac{112}{11} & -\frac{147}{11} & \frac{4038}{121} & \frac{1410}{121} & \frac{555}{242} & \frac{13921}{2662} & -\frac{116264}{14641} \\
 9 & \frac{305}{11} & \frac{70}{11} & \frac{134433}{1694} & \frac{6883}{242} & \frac{1410}{121} & \frac{62803791}{260876} & \frac{800410}{9317} & \frac{202793}{5324} & \frac{293}{484} \\
 3 & \frac{70}{11} & \frac{112}{11} & \frac{6883}{242} & \frac{1410}{121} & \frac{555}{242} & \frac{800410}{9317} & \frac{202793}{5324} & \frac{293}{484} & \frac{444431}{14641} \\
 1 & \frac{112}{11} & -\frac{147}{11} & \frac{1410}{121} & \frac{555}{242} & \frac{13921}{2662} & \frac{202793}{5324} & \frac{293}{484} & \frac{444431}{14641} & -\frac{36339363}{644204} \\
 1 & -\frac{147}{11} & \frac{4038}{121} & \frac{555}{242} & \frac{13921}{2662} & -\frac{116264}{14641} & \frac{293}{484} & \frac{444431}{14641} & -\frac{36339363}{644204} & \frac{930018189}{7086244} \\
\end{array}
\right).$$
\end{small}
The calculation shows  that $ \operatorname{rank} \mathcal{M}(3) = \operatorname{rank} \mathcal{M}(2)= 4 $, i.e. $ \mathcal{M}(3) $ is a flat extension of $ \mathcal{M}(2)$.

\end{document}